\documentclass[12pt]{amsart}
\usepackage{amssymb}
\usepackage{amsmath, amscd}
\usepackage{amsthm}
\usepackage{xcolor}
\usepackage{ulem}

\newtheorem{theorem}{Theorem}[section]

\newtheorem{proposition}[theorem]{Proposition}
\newtheorem{lemma}[theorem]{Lemma}

\newtheorem{corollary}[theorem]{Corollary}
\theoremstyle{definition}

\newtheorem{example}[theorem]{Example}
\newtheorem{definition}[theorem]{Definition}

\newtheorem{claim}[theorem]{Claim}

\newcommand{\N}{\mathbb{N}}
\newcommand{\Z}{\mathbb{Z}}

\def\b{\beta}

\def\endo#1#2#3#4{\left[\begin{matrix}
                           #1 & #2 \cr
                           #3 & #4
                        \end{matrix}\right]}

\usepackage[utf8]{inputenc}
\usepackage{amssymb}
\usepackage{amsmath, amscd}
\usepackage{amsthm}
\usepackage{xcolor}

\begin{document}

\author[P.V. Danchev]{Peter V. Danchev}
\address{Institute of Mathematics and Informatics, Bulgarian Academy of Sciences, 1113 Sofia, Bulgaria}
\email{danchev@math.bas.bg; pvdanchev@yahoo.com}
\author[P.W. Keef]{Patrick W. Keef}
\address{Department of Mathematics, Whitman College, Walla Walla, WA 99362, USA}
\email{keef@whitman.edu}

\title[Abelian $p$-Groups with Minimal Characteristic Inertia] {Abelian $p$-Groups with \\ Minimal Characteristic Inertia}

\maketitle

\begin{abstract} For Abelian $p$-groups, Goldsmith, Salce, et al., introduced the notion of {\it minimal full inertia}. In parallel to this, we define the concept of {\it minimal characteristic inertia} and explore those $p$-primary Abelian groups having minimal characteristic inertia. We establish the surprising result that, for each Abelian $p$-group $A$, the square $A\oplus A$ has the minimal characteristic inertia if, and only if, it has the minimal full inertia. We also obtain some other relationships between these two properties. Specifically, we
exhibit groups which do not have neither of the properties, as well as we show via a concrete complicated construction that, for any prime $p$, there is a $p$-group possessing the minimal characteristic inertia which does not possess the minimal full inertia.
\end{abstract}

\vskip2.0pc

\section{Introduction and Conventions}

All groups considered will be Abelian $p$-primary for some arbitrary but a fixed prime $p$. Our notation and terminology will generally agree with the classical books \cite{F1}, \cite{F2}, \cite{G} and \cite{Kap}. If $X$ is a set, $|X|$ will denote its cardinality, and if $x$ is an element of a group, $|x|$ will denote its $p$-height.

\medskip
	
As part of a more general investigation into the concept of {\it algebraic entropy} (see, for example, \cite{GS:entropy}), Goldsmith,  Salce, et al., introduced the important concept of {\it minimal full inertia} (see \cite{GSZ} and \cite{GS}). In particular, they proved some significant results like these: {\it the direct sums of cyclic groups have minimal full inertia} (see \cite {GS} Corollary 3.3) as well as that {\it the class of groups which have minimal full inertia is not closed under taking finite direct sums} (see \cite {GS} , Example 3.6). They also  construct numerous groups that either do or do not have minimal full inertia (see \cite {GS}, Proposition 3.7, Example 3.8 and Proposition 4.1, respectively).

\medskip

Furthermore, the fact that direct sums of cyclic groups have minimal full inertia was greatly generalized in \cite{K} where a class of groups properly containing the totally projective groups (just called ``countably totally projective") was also shown to have this property.

These ideas utilize ideas that go back at least as far as \cite{K}. The collection of endomorphisms of a group are used to define when a subgroup is fully invariant. If the collection of endomorphisms is replaced by the set of automorphisms of the group, the parallel notion is that of a characteristic subgroup.  The central theme of this paper is to investigate how this change of perspectives affects the property of when a group has minimal inertia.

\medskip

We turn to formally defining the above terms.

\begin{definition} Two subgroups of a given group are called {\it commensurable}, provided their intersection has finite index in each of them.
\end{definition}

Throughout the text, if $X$ is a  subgroup of a group $G$ and $\phi$ is an endomorphism of $G$, then let $\hat \phi(X)=(\phi(X)+X)/X$. So, $X$ is {\it characteristic} or {\it fully invariant} in $G$ if, and only if, $\hat \phi(X)=0$ for every automorphism or endomorphism, respectively, of $G$. It is obvious and well-known that fully invariant subgroups are always characteristic, while the converse implication fails. It is then of some interest to consider those groups whose characteristic subgroups are fully invariant and, even more generally, when all characteristic subgroups are commensurable with fully invariant subgroups. Thus the objective of our work is to examine exactly such groups. To that goal, the following notion appeared in \cite{CDG}.

\begin{definition}
A subgroup $X$ of $G$  is {\it characteristically inert} (or {\it fully inert}) if $\hat\phi(X)$ is finite for every automorphism (respectively, endomorphism) of $G$.
\end{definition}

It is easy to check that, if the subgroup $X$ is commensurable with a characteristic or fully invariant subgroup of the whole group $G$, then it is too characteristically or fully inert, respectively.

Continuing our parallel between characteristic versus fully invariant and characteristically versus fully inert, we are led to the following idea.

\begin{definition}
The group  $G$ has {\it minimal characteristic inertia} (or {\it minimal full inertia}) if every characteristically inert subgroup (respectively, fully inert subgroup) is commensurable with some characteristic (respectively, fully invariant) subgroup of $G$.
\end{definition}

Our further work is organized thus: In the next section, we formulate and prove our main assertions on groups with minimal characteristic inertia and some closely related properties (see, e.g., Theorems~\ref{mincharfull} and \ref{big}). In the subsequent section, we construct a series of examples illustrating these ideas (see, e.g., Examples~\ref{noone}, \ref{minchar} and \ref{transnonfull}).  We finish our exposition by stating two unresolved problems of some interest and importance which, hopefully, will stimulate a further investigation of the subject (see Problems 1 and 2).

\section{Statements and Proofs}

We begin with a simple but useful observation.

\begin{lemma}\label{bound}
Suppose $X$ is a subgroup of $G$ and $\phi_1, \dots, \phi_k$ are endomorphisms of $G$ such that $\hat\phi_i(X)$ is finite for $i=1,\dots, k$. If $\gamma = \phi_1+\cdots \phi_k$, then $\hat \gamma (X)$ is finite and its cardinality satisfies the inequality
$$          |\hat \gamma(X)|\leq|\hat \phi_1(X)| \cdots |\hat\phi_k(X)|.
$$
\end{lemma}

\begin{proof} Simple addition gives a homomorphism
$$\sigma: Z:=\hat \phi_1(X)\oplus \cdots \oplus \hat \phi_k(X)\to G/X.
$$
The condition  $\gamma = \phi_1+\cdots \phi_k$ easily implies that $\hat \gamma(X)\subseteq \sigma(Z)$.  Therefore,
$$
            |\hat \phi_1(X)| \cdots |\hat\phi_k(X)|=\vert Z\vert \geq \vert \sigma(Z)\vert \geq \vert \hat \gamma(X)\vert,
$$
as required.
\end{proof}

\begin{proposition}\label{next} Suppose $G$ has the property that every endomorphism of $G$ is the sum of automorphisms, and $X$ is a subgroup of $G$.

(a) $X$ is fully invariant if, and only if, it is characteristic.

(b) $X$ is commensurable with a fully invariant subgroup if, and only if, it is commensurable with a characteristic subgroup.

(c) $X$ is fully inert if, and only if, it is characteristically inert.
\end{proposition}

\begin{proof}
The forward direction in all these equivalences follows immediately from the definitions.

If $\gamma$ is any endomorphism of $G$, then $\gamma = \phi_1+\cdots +\phi_k$, where each $\phi$ is an automorphism. So, assuming $X$ is characteristic, each $\hat \phi_i(X)=0$, which by Lemma~\ref{bound}, implies that $\hat \gamma(X)=0$. Letting $\gamma$ range over all endomorphisms of $G$ gives the converse in (a). And clearly, the converse in (b) follows from the converse in (a).

Similarly, if $X$ is characteristically inert, $\hat \phi_i(X)$ is finite for each $i=1, \dots, k$. So, again by Lemma~\ref{bound}, $\hat \gamma(X)$ will always be finite, as required.
\end{proof}

By letting $X$ range overall subgroups of $G$, we obtain the following result.

\begin{corollary}\label{frog}
If $G$ has the property that every endomorphism of $G$ is the sum of automorphisms, then $G$ has minimal full inertia if and only if it has minimal characteristic inertia. And in this case, for any subgroup $X$ of $G$, the two conditions in Proposition~\ref{next}(b) are logically equivalent to the two conditions in Proposition~\ref{next}(c).
\end{corollary}

We recall the following result, due to Paul Hill:

\begin{theorem}[\cite{H}, Theorem~4.2]\label{Hill} A totally projective $p$ group, $p\ne 2$, has the property that any endomorphism is the sum of two automorphisms.
\end{theorem}

In \cite{K} it was shown that any totally projective group has minimal full inertia.  Putting this together with Hill's result leads to the next observation.

\begin{corollary}\label{totproj}
If $G$ is a totally projective group $p$-group with $p\ne 2$, then $G$ has minimal characteristic inertia and, for a subgroup $X$ of $G$, the four statements in Proposition~\ref{next} (b) and (c) are equivalent.
\end{corollary}

The next result is often attributed to Kaplansky (see, e.g., \cite{Kap} as well as the proof of \cite[Theorem 4.3]{CDG}), however for the reader's convenience and completeness of the exposition we shall provide a simple proof.

\begin{lemma}\label{sum}
Suppose $A$ is a group and $G=A\oplus A$. Then any endomorphism of $G$ is the sum of at most four automorphisms.
\end{lemma}

\begin{proof}
Suppose $R$ is the endomorphism ring of $A$ and $1_A\in R$ is the identity. Writing homomorphisms on the right, any endomorphism $G$ can be thought of as multiplication on the right by a 2 by 2 matrix with entries in $R$. For any such endomorphism we have a decomposition:

\medskip

$$
\endo \alpha \beta \delta \epsilon=\endo {1_A} \beta 0 {1_A} + \endo {-1_A} 0 \delta  {-1_A}+
           \endo \alpha {1_A} {1_A} 0  + \endo 0 {-1_A} {-1_A} \epsilon.
$$

\medskip

\noindent Clearly, the last four matrices represent automorphisms of $G$, as expected.
\end{proof}

In particular, this means that Corollary~\ref{frog} applies to any such ``squared" group.

We now arrive at the following curious assertion, which actually somewhat extends \cite[Lemma 3.4]{GS} and is in parallel to a well-known result from \cite{FG}.

\begin{theorem}\label{mincharfull} If $A$ is a group and $G=A\oplus A$, then the following three points are equivalent:

(a) $A$ has minimal full inertia.

(b) $G$ has minimal full inertia;

(c) $G$ has minimal characteristic inertia.
\end{theorem}

\begin{proof} The equivalence of (b) and (c) is an immediate consequence of Corollary~\ref{frog} and Lemma~\ref{sum}. To show they are equivalent to (a) we first fix some notation. Let $\pi_1:G\to A$ be the projection onto the first summand and $\rho_1:A\to G$ be the obvious inclusion into the first summand; define $\pi_2$ and $\rho_2$ similarly.  We also let $\kappa_i=\rho_i\circ \pi_i:G\to G$  ($i=1,2$) be the corresponding idempotent endomorphisms and $\sigma:G\to G$ be the automorphism given by $\sigma((a_1,a_2))=(a_2,a_1)$.

We first assume (b) holds. To prove (a), suppose $Y$ is a fully inert subgroup of $A$. We claim that $X:=Y\oplus Y$ is fully inert in $G$. If $\gamma$ is an endomorphism of $G$, and for $i,j=1,2$ we let $\gamma_{i,j}=\pi_i\circ\gamma\circ \rho_j:A\to A$, it is clear that
$$
          \gamma(X)\subseteq [(\gamma_{1,1}) (Y))+(\gamma_{1,2})(Y))]\oplus [(\gamma_{2,1}) (Y))+(\gamma_{2,2})(Y))]:=L
$$
For $i,j\in \{1,2\}$, since $Y$ is fully inert in $A$, it follows that $\hat\gamma_{i,j}(Y)$ will be finite.
Adding in the two summands clearly determines a surjection
$$
         (\hat\gamma_{1,1}(Y)\oplus  \hat\gamma_{1,2}(Y))\oplus   (\hat\gamma_{2,1}(Y)\oplus  \hat\gamma_{2,2}(Y)) \twoheadrightarrow [L+X]/X.
$$
Therefore, $[L+X]/X$ must be finite, so that $\hat\gamma(X)\subseteq [L+X]/X$ is also finite.

Letting $\gamma$ range over all endomorphisms of $G$, we conclude that $X$ is fully inert in $G$. So, by hypothesis, $X\sim W$, where $W$ is fully invariant subgroup in $G$. Since $\kappa_i(W)\subseteq W$ ($i=1,2$) and $\sigma (W)\subseteq W$, we can conclude that $W=V\oplus V$, where $V=\pi_i(W)$ is fully invariant in $A$. It readily follows that $Y\sim V$, giving the result.

For the converse, suppose $A$ has minimal full inertia. Let $X$ be any fully inert subgroup of $G$. Note that if
$$
           Y_1=\kappa_1(X)+\sigma(\kappa_2(X))\subseteq A\oplus 0, \ \ {\rm and}\ \  Y_2=\kappa_2(X)+\sigma(\kappa_1(X))\subseteq 0\oplus A,
$$
then $\sigma$ restricts to an isomorphism between $Y_1$ and $Y_2$; let $Y=\pi_1(Y_1)=\pi_2(Y_2)\subseteq A$.
Since $X\subseteq \kappa_1(X)+\kappa_2(X)$, we can conclude that $X\subseteq Y_1+Y_2= Y\oplus Y$. And since $X$ is fully inert, for $i=1,2$, both groups
$$
          R_i:=\hat \kappa_i(X)=   (\kappa_i(X)+X)/X    \ \ {\rm and}\  \             S_i:=\hat {\sigma\circ \kappa}_i(X)=  (\sigma(\kappa_i(X))+X)/X
$$
are finite. Since addition in each summand gives a surjective homomorphism
$$
        (R_1\oplus S_2)\oplus (R_2\oplus S_1) \twoheadrightarrow (Y_1 +Y_2)/X = (Y\oplus Y)/X,
$$
we can conclude that $(Y\oplus Y)/X$ is finite; i.e., $X\sim Y\oplus Y$. Now, since $X$ is fully inert in $G$, $Y\oplus Y$ will be, as well. And from this, we can easily conclude that $Y$ is fully inert in $A$. Therefore, there is a fully invariant subgroup $V\subseteq A$ such that $Y\sim V$. It follows that $W:=V\oplus V$ is fully invariant in $G$ and $X\sim Y\oplus Y\sim V\oplus V=W$. Thus, $G$ has minimal full inertia, completing the proof.
\end{proof}

Note that the equivalence of conditions (a) and (b) is clearly related to Lemma~3.4 of \cite{GS}, where the groups were assumed to be fully transitive.

\medskip

In regard to the last theorem and \cite[Theorem 3.5]{GS}, a question which directly arises is what can be said for the group $G$ being an infinite direct sum of copies of the group $A$?

\medskip

We now consider when the one property implies the other. The following observation shows that in one important case, the property of having minimal characteristic inertia is stronger than having minimal full inertia.

\begin{proposition}\label{minfull}
Suppose $G$ has the property that every characteristic subgroup of $G$ is fully invariant. If $G$ has minimal characteristic inertia, then it has minimal full inertia.
\end{proposition}

\begin{proof}
Suppose $X$ is a fully inert subgroup of $G$; we need to show that it is commensurable with a fully invariant subgroup.

Certainly, $X$ is characteristically inert, as well. Therefore, there is a characteristic subgroup $Y\subseteq G$ that is commensurable with $X$. By hypothesis, we know that $Y$ will also be fully invariant in $G$, completing the argument.
\end{proof}

The last result has an important special case.

\begin{corollary}
Suppose $p$ is odd and $G$ is transitive. If $G$ has minimal characteristic inertia, then it has minimal full inertia.
\end{corollary}

\begin{proof}
It is well known that, when $p$ is odd and $G$ is transitive, a characteristic subgroup must be fully invariant.  Indeed, suppose $X$ is characteristic, $\gamma$ is an endomorphism of $G$ and $x\in X$. The height sequences satisfy the inequality $\Vert x\Vert\leq \Vert \gamma(x)\Vert$. By the argument of (\cite {K}, Theorem~26), this implies that $\phi(x)=y_1+y_2$, where $\Vert x\Vert=\Vert y_1 \Vert =\Vert y_2 \Vert$. So, there are automorphisms $\phi_i$ ($i=1,2$) such that $y_i=\phi_i(x)\in X$. Thus, $\gamma (x)=y_1+y_2\in X$, as required.
\end{proof}

Since separable groups are always transitive, we have the following consequence.

\begin{corollary}\label{sep}
Suppose $p$ is odd and $G$ is separable. If $G$ has minimal characteristic inertia, then it has minimal full inertia.
\end{corollary}

So, if we want to find a group $G$ that has minimal full inertia, but not minimal characteristic inertia, we must have either $p=2$ or $p^\omega G\ne 0$. For instance, if $B=\oplus_{n\in \mathbb N} \mathbb Z_{2^n}$ is the standard direct sum of cyclic $2$-groups, then it is known with the aid of \cite{GS} that both $B$ and $\overline B$ has minimal full inertia. So, if either fails to have the minimal characteristic inertia, then we would have our wanted example.

We end this discussion with one last observation which expresses the question entirely in terms of minimal characteristic inertia.

\begin{corollary}
The group $A$ has minimal full inertia, but not minimal characteristic inertia if, and only if, $G=A\oplus A$ has minimal characteristic inertia, but $A$ itself does not have minimal characteristic inertia.
\end{corollary}

Using techniques from \cite {K}, our next goal is to verify that every separable group is a summand of a group with both minimal characteristic inertia and minimal full inertia.  The following easy result from that work will be useful.

\begin{lemma}[\cite K, Lemma~2.3] \label{quotient}Suppose $L$ is a group and $B, C$ are subgroups of $L$.

(a) If $B\subseteq C$ and $k<\omega$, then $B\sim C$ if, and only if, $B[p^k]\sim C[p^k]$ and $p^k B\sim p^k C$.

(b) If $n<\omega$ and $B[p^{n}]\subseteq C[p^{n}]$, then $B[p^{n}]\sim C[p^{n}]$ if, and only if, $(p^k B)[p]\sim (p^k C)[p]$ for all $k< n$.
\end{lemma}

If $H$ is a group, then any endomorphism or automorphism on $H$ restricts to an endomorphism or automorphisms, respectively, on $p^\omega H$. The following construction applies when all of these restricted automorphisms are simply multiplications, which implies that every subgroup of $p^\omega H$ will be characteristic in $H$.

\begin{theorem}\label{big} Suppose $H$ is a group such that $p^{\omega+1} H=0$ and every automorphism of $H$ when restricted to $p^\omega H$ is multiplication by some (non-zero) element of $\Z_p$. Then there is a separable group $K$ such that the group $G=H\oplus K$ has minimal characteristic inertia.
\end{theorem}

\begin{proof} If $H$ is bounded, then $H\oplus H$ is a direct sum of cyclic groups, and hence this square has minimal full inertia (cf. \cite{GS}, \cite{K}). Therefore, according to Theorem~\ref{mincharfull}, it also has minimal characteristic inertia, and thus we can just let $K=H$. So, without loss of generality, we may assume that $H$ is unbounded and, in particular, that it is infinite.

Let $\kappa>|H|^{\aleph_0}$ be some cardinal with $\kappa^{\aleph_0}=\kappa$. Next, let $B_1$ be a direct sum of cyclic groups all of whose Ulm invariants (at finite ordinals) are equal to $\kappa$; so the torsion completion, $\overline B_1$, will also have cardinality $\kappa$. Let $B_2$ be a direct sum of cyclic groups, all of whose (again, finite) Ulm invariants equal $\aleph_1$. We let $K=\overline B_1\oplus B_2$ and show that $G:=H\oplus \overline B_1\oplus B_2=H\oplus K$ has the required properties.

Define an ordered set $\mathcal O$ as follows: $\mathcal O$ is the union of $\omega=\{0,1,2,\dots\}$, and
$\mathcal S$, the set of subgroups of $M:=p^\omega H$. We identify  the symbol $\infty$ with the zero subgroup. The elements of  $\omega\subseteq \mathcal O$ will (naturally) be called integers and  the others will (again, naturally) be called subgroups. We define $\alpha<\beta$ as follows: if $\alpha$ and $\beta$ are integers and this is true in the usual sense; if $\alpha$ is an integer and $\beta$ is a subgroup and; if $\beta=\infty$.  If $\alpha\in \mathcal S$, we agree that $\alpha+1=\infty$. If $E\in \mathcal S$, we define $p^E G=E$; if $E$ is $\infty=\{0\}$, this agrees with the usual definition. For each $\beta\in \mathcal O$ let  $S_\beta=(p^\beta G)[p]$. If $\beta\in \mathcal O$ and $X$ is a subgroup of $G$, let $X(\beta)=X\cap p^\beta G$ and $X_\beta= X/X(\beta)$; if $\beta$ is an integer or $\infty$, these agree with the usual definitions. Note that $X(M)$ agrees with the more usual $X(\omega)$ and $X_M$ with the more usual $X_\omega=X/X(\omega)$.

\medskip

We will show that any characteristically inert subgroup is commensurable with a characteristic subgroup of a particular form. If $\alpha_0<\alpha_1<\alpha_2<\cdots$ are in $\mathcal O$ and $\overline \alpha =(\alpha_0, \alpha_1, \dots)$, let $G(\overline \alpha)$ be the set of $x\in G$ such that for all $n<\omega$, $p^n x\in p^{\alpha_n} G$.  So, for example, if $E\in \mathcal O$ is a subgroup, then
$G(1, 3, E, \infty, \infty, \dots)$ is the collection of all $x$ such that $|x|\geq 1$, $|px|\geq 3$, $p^2 x\in E$ and $p^3 x=0$.

Since all $E\in \mathcal S$ are characteristic, it is relatively straightforward to verify that any subgroup of the form $G(\overline \alpha)$ will be characteristic, as well.

\medskip

Our arguments will be based on the following technical observation which will provide us with an abundance of automorphisms to use. It will play a role similar to that of Lemmas~2.15 and~2.16 in \cite K.

\begin{claim} \label{sheer}Suppose $X\subseteq G$ is a subgroup, $G=V\oplus W$ is a decomposition and $\gamma:V\to W$ is a homomorphism. Let $\phi_\gamma$ be the automorphism of $G$ given by $\phi_\gamma(v+w)=v+\gamma(v) +w$ for all $v\in V$, $w\in W$.

(a) If for all $j<\omega$, $v_j\in V\cap X$ with  $w_j:=\gamma(v_j)\in W\subseteq G$ and the elements $w_j+X\in G/X$ are distinct, then $\hat \phi_\gamma(X)$ is infinite, so that $X$ is not characteristically inert.

(b) If $\{n_j\}_{j<\omega}$ is a strictly increasing sequence of integers and for all $j<\omega$, $v_j\in V\cap X$ with  $w_j:=\gamma(p^{n_j}v_j)\in W[p]\setminus p^{n_j} X$, then $\hat \phi_\gamma(X)$ is infinite, so that $X$ is not characteristically inert.
\end{claim}

Clearly, $\phi_\gamma$ is an automorphism (its inverse is $\phi_{-\gamma}$.)

Regarding (a), consider the elements of $\hat \phi_\gamma(X)$ of the form
$$
               \phi_\gamma(v_j)+X =v_j+w_j+X= w_j+X.
$$
Since we are assuming these are distinct, the result follows.

For (b), suppose $j>k$ and $\phi_\gamma(v_{n_j})$ and $\phi_\gamma(v_{n_k})$ represent the same element of $\hat \phi_\gamma(X)$.  So  $\phi_\gamma(v_{n_j})=v_{n_j}+\gamma(v_{n_j})$ and $\phi_\gamma(v_{n_k})=v_{n_k}+\gamma(v_{n_k})$ are congruent modulo $X$. That is, $\gamma(v_{n_j})\in W$ and $\gamma(v_{n_k})\in W$ are congruent modulo $X$. Therefore,
$$
             w_j=w_j-p^{n_j-n_k} w_k= p^{n_j}(\gamma(v_{n_j})-\gamma(v_{n_k}))\in p^{n_j} X
$$
contrary to hypothesis.

\vskip .2in

Throughout, we will let $X$ be some characteristically inert subgroup of $G$; our goal is to construct a sequence $\overline \alpha$ such that $X\sim G(\overline \alpha)$.

\begin{claim}\label{cases} One of three things happens:

\medskip

(a)  $X$ is finite. In this case, $X\sim 0$ and we let $\alpha_X=\infty\in \mathcal O$.

\medskip

(b) $X$ is infinite and $X_\omega$ is finite.  In this case,  $X\sim X(\omega)\in \mathcal S\subseteq \mathcal O$ and we let $\alpha_X=X(\omega)$.

\medskip

(c) $X_\omega$ is infinite.  In this case there is an integer $\alpha_X\in \mathcal O$ such that $S_{\alpha_X}\sim X(\alpha_X)[p]$ and  $X\sim X(\alpha_X)$, i.e., $S_{\alpha_X}/X(\alpha_X)[p]$ and $X_{\alpha_X}$ are finite. (c.f., \cite K, Lemma~2.7)
\end{claim}

Clearly, exactly one of the three conditions holds, and the conclusions in (a) and (b) are routine. So assume  $X_\omega$ is infinite. Note that $|X(\omega)|\leq |H|<\kappa$. We first show that the cardinality of $|X_\omega|=|X|=\kappa$. Otherwise, there is a decomposition $\overline B_1=\overline B_3\oplus W$, where $\overline B_3$ and $W$ are copies of $\overline B_1$, and in the decomposition $G=H\oplus \overline B_3\oplus W\oplus B_2$, we have
$
                   X\subseteq V:=H\oplus \overline B_3\oplus 0\oplus B_2.
$

If $\overline V_\omega$ is the torsion-completion of $V_\omega$, then there is an isomorphism $\overline V_\omega\to W$. Let $\gamma$ be the composition of the natural map $V\to V_\omega\to \overline V_\omega$ with this isomorphism. If we choose any sequence of element $\{v_j\}_{j<\omega}$ of $X$ that are pairwise not congruent modulo $X(\omega)$, then the elements $w_j=\gamma(v_j)$ are distinct elements of $W$. So if $j<k$, then $w_j-w_k$ is a non-zero element of $W$, so in particular, it is not in $X\subseteq V$.  Since $X$ is characteristically inert, this contradicts Claim~\ref{sheer}(a), showing that $X_\omega$ must have cardinality $\kappa$, as claimed.

\medskip

There is a valuated decomposition $X[p]\cong X_h\oplus X(\omega)$. Since $|X[p]|=\kappa$ and $|X(\omega)|<\kappa$, $X_h$ must have cardinality $\kappa$.  Now, $X_h$ is a separable valuated vector space, so it has a basic subspace. If each of its (finite) Ulm invariants were finite, then $X_h$ (which is contained in a completion of this basic subspace) would also have cardinality at most $2^{\aleph_0}<\kappa$ contrary to what we just established.  Choose the integer $\alpha_X$ minimal with the property that the $\alpha_X$th Ulm  invariant of $X_p$ is infinite. For simplicity, we will just denote $\alpha_X$ by $\alpha$; we need to show $S_\alpha/ X(\alpha)[p]$ and $X/X(\alpha)$ are finite.

\medskip

Suppose that $S_\alpha/ X(\alpha)[p]$  is infinite. Since the $\alpha$th Ulm invariant of $X_h$ is infinite, we can construct a  decomposition $G=V\oplus W$,  where $V\cong \oplus_{j<\omega}\langle v_j'\rangle\cong \oplus_{j<\omega}\Z_{p^{\alpha+1}}$ such that $V[p]\subseteq X$. Since $S_\alpha/ X(\alpha)[p]$  is infinite, we can also find elements $\{w_j\}_{j<\omega}$ of $ W(\alpha)[p]$ that are pairwise not congruent modulo $X(\alpha)[p]$, and hence not congruent modulo $X$. If $v_j=p^\alpha v_j'\in V[p]\subseteq X$, there is clearly a homomorphism $\gamma:V\to W$ such that $\gamma(v_j)=w_j$ for $j<\omega$. So by Claim~\ref{sheer}(a) we can conclude that $X$ is not characteristically inert, contrary to hypothesis.

\medskip We now need to show that $X/X(\alpha)$ is finite, so assume it is not. Let $\beta< \alpha$ have the property that $X(\beta)/X(\beta+1)$ is infinite; such a value must clearly exist. Let $\{v_j\}_{j<\omega}\subseteq X(\beta)$  be linearly independent modulo $X(\beta+1)$.

There is clearly a decomposition $G=V\oplus W$ where  $ \{v_j\}_{j<\omega}\subseteq  V$; and the $\beta$th Ulm invariant of $W$ is infinite.

Note that $V_{\beta+1}=V/p^{\beta+1}V$ is $p^{\beta+1}$ bounded. This easily implies that there is a decomposition $V_{\beta+1}=Y\oplus Y'$ for which $Y[p]=\langle v_j+p^{\beta+1}V:j\in \omega\rangle$.
Since $\alpha$ is the first infinite Ulm invariant of $X_h$, we can conclude that
the image of $X(\beta)[p]/X(\beta+1)[p]\to S_\beta/S_{\beta+1}$ is finite.   On the other hand, the image of $p^\beta W[p]/p^{\beta+1}W[p]\to S_\beta/S_{\beta+1}$ is infinite.
Therefore, there are elements $w_j\in (p^\beta W)[p]$ that are pairwise not congruent modulo $X(\beta)[p]$, and hence pairwise not congruent modulo $X$.

Clearly, $v_j+ p^{\beta+1}V\mapsto w_j$  can be extended first to a homomorphism $\lambda: Y\to W$ and then, by setting $\lambda(Y')=0$, to a homomorphism $\lambda: V_{\beta+1}\to W$.
Let $\gamma$ be the composition
$
                V\to V_{\beta+1} \to  W.
$
As before, the elements $ \gamma(v_j)=\lambda (v_j+p^{\beta+1} V)= w_j\in \lambda (Y)$ are not congruent modulo $X$.  Since $X$ is characteristically inert, this contradicts Claim~\ref{sheer}(a), completing the proof of Claim~\ref{cases}.

\vskip .2in

We now define an entire sequence $\overline \alpha$ for $X$ (c.f., \cite K, Theorem~2.11).
For each $n<\omega$, $p^n X$ will also be characteristically inert in $G$ and we let $\alpha_n=\alpha_{p^n X}$ (in fact, this is just a simplification of notation).   We first verify that $\alpha_n+1\leq \alpha_{n+1}$. If $\alpha_n=\infty$, then $p^n X$ is finite, which implies that $p^{n+1} X$ is finite, so that $\alpha_{n+1}=\infty$, as required. Similarly, if $\alpha_n$ is a subgroup, then $(p^nX)_\omega$ is finite. Since $(p^nX)(\omega)\subseteq (p^n X)[p]$ this implies that $p^{n+1} X\cong (p^n X)/(p^n X)[p]$ is also finite, Therefore, again, $\alpha_{n+1}=\infty$, as required.

Finally, suppose $\alpha_n$ is an integer. If $\alpha_{n+1}$ is a subgroup, then the result is trivial.  So assume $\alpha_n$, $\alpha_{n+1}$ are integers.  Since $p^n X/p^n X(\alpha_n)$ is finite and multiplication by $p$ gives a surjection onto  $p^{n+1} X/p^{n+1} X(\alpha_n+1)$, the latter quotient is also finite.  And since $\alpha_{n+1}$ is the largest integer such that $p^{n+1} X/p^{n+1} X(\alpha_{n+1})$ is finite, we must have $\alpha_n+1\leq \alpha_{n+1}$.

\begin{claim}\label{zeroclaim} (cf. \cite K, Theorem~2.11(4))
For all $n<\omega$:
$$X\sim X_n:=X\cap G(\alpha_0, \alpha_1, \alpha_2, \dots, \alpha_{n-1}, \alpha_{n-1}+1, \alpha_{n-1}+2, \dots).$$
\end{claim}

If $n=1$, this just says $X\sim X(\alpha_X)$, which we know is true. We now show that for $n\geq 1$ we have $X_{n}\sim X_{n+1}$, which will complete the argument. Consider the homomorphism $\nu$
given by the composition
$$
          X_{n}  \,\smash{\mathop{\longrightarrow}\limits^{\times p^{n}}}\,   p^{n} X_{n}\subseteq p^{n} X \to (p^{n} X)/(p^{n} X)(\alpha_{p^{n} X}).
$$
It readily follows that $X_{n+1}$  is the kernel of $\nu$. And since we know this last group is finite, we have $X_{n}\sim X_{n+1}$, as desired.

\vskip .2in
Let $A=G(\overline \alpha)$. Our goal is to show $X\sim A$.

\begin{claim} \label{firstclaim} (cf. \cite K, Theorem~2.11(5)) If $n<\omega$, then $X[p^n]\sim A[p^n]$. \end{claim}

We know that $X_n[p^n]\subseteq A[p^n]$. For all $0\leq k<n$ we have $p^kX_n[p]\sim p^k X[p]\sim S_{\alpha_n}= p^k A[p]=p^k (A[p^n])[p]$. So by Lemma~\ref{quotient}(b), $X_n[p^n]\sim A[p^n]$, which implies that $X[p^n]\sim A[p^n]$.

\vskip .2in

Note that if $\alpha_n=\infty$ for some $n<\omega$, then it follows that $X\sim X_{n+1}=X_{n+1}[p^{n}]\sim A[p^n]=A$ and the result follows.

So, from here on, we may assume $\alpha_n$ is an integer for all $n$.

\begin{claim}\label{secondclaim} There is an $N<\omega$ such that for every $n\geq N$ we have $p^n X\subseteq p^{\alpha_n}G$:
\end{claim}

We suppose this fails and derive a contradiction. So, there must exist a strictly ascending sequence $n_0<n_1<n_2<\cdots$ and elements $v_{n_j}\in X$ such that $|p^{n_j} v_{n_j}|<\alpha_{n_j}$.

Find a decomposition $\overline B_1=\overline B_3\oplus W$ such that for all $j<\omega$ we have  $v_{n_j}\in V:=H \oplus \overline B_3\oplus 0 \oplus B_2$ and each Ulm factor of $W$ is infinite.

For $j<\omega$, let $v_{n_j}'=v_{n_j}+p^\omega V\in V_\omega$; so each $p^{n_j}v'_{n_j}\ne 0$. After possibly restricting to a subsequence, we may assume that for all $j<\omega$, if $n_j\leq k\in \omega$ and $p^k v'_{n_j}\ne 0$, then $| p^k v_{n_j}'|<|p^{n_{j+1}}v_{n_{j+1}}'|$. Using this condition, it can be checked that there is a basic subgroup of $V_\omega$ the form $C=\oplus_{j<\omega} C_j$ such that for each $j<\omega$ we have $p^{n_j}v_{n_j}'\in C_j$.

For each $j<\omega$, if $\b=|p^{n_j}v'_{n_j}|<\alpha_{n_j}$, then the image of
$$p^{n_j} X(\beta)[p]/p^{n_j} X(\beta+1)[p]\to S_\beta/S_{\beta+1}$$ is finite.   On the other hand, the image of
$$p^\beta W[p]/p^{\beta+1}W[p]\to S_\beta/S_{\beta+1}$$ is infinite. So we can find a $w_j\in W[p]\setminus p^{n_j} X$ with $|w_j|=\beta=|p^{n_j}v'_{n_j}|$. The assignment $p^{n_j}v'_{n_j}\mapsto w_j$ clearly extends to a homomorphism $C_j\to W$. Summing extends these all to a homomorphism $C\to W$. And since $C$ is pure in $V_\omega$ and $W$ is pure-injective for torsion groups, this extends to a homomorphism $V_\omega \to W$. We let $\gamma:V\to W$ be the composition $V\to V_\omega \to W$, so that $\gamma(p^{n_j}v_{n_j})=w_{n_j}$.

Since $X$ is characteristically inert, this contradict Claim~\ref{sheer}(b) and establishes the claim.

\vskip .2in

It follows from this that $X\sim X_N\subseteq A$.  So, replacing $X$ by $X_N$ there is no loss of generality in assuming that $X\subseteq A$.

\begin{claim}\label{thirdclaim} There is an $N<\omega$ such that for all $n\geq N$ we have $(p^n X)[p]=S_{\alpha_n}=(p^n A)[p]$.
\end{claim}

If the claim fails, then we can find a strictly increasing sequence of elements of $\omega$, $n_0<n_1<n_2<\cdots$ and elements $w_{n_j}\in (p^{n_j} A)[p]\setminus (p^{{n_j}} X)$. There is clearly a decomposition $B_2=B_4\oplus V$ such that every Ulm factor of $V$ is infinite and for each $j<\omega$ we have $w_{n_j}\in H\oplus \overline B_1 \oplus B_4\oplus 0:=W$.

Since by Claim~\ref{firstclaim}, $X[p^{n_j+1}]\sim A[p^{n_j+1}]$, we can conclude that
$$(X\cap V)[p^{n_j+1}]\sim V(\overline \alpha)[p^{n_j+1}].$$
Consequently,
$$
                                p^{n_j}((X\cap V)[p^{n_j+1}])\sim p^{n_j}( V(\overline \alpha)[p^{n_n+1}])= (p^{\alpha_{n_j}}V)[p].
$$
Therefore, we can find $v_j\in X\cap V$ such that $p^{n_j} v_j\in V[p]$ and $|p^{n_j} v_j|=\alpha_{n_j}$.

Again, $|p^{n_j} v_j|$ goes to infinity as $j$ does. So if we start with a decomposition of $V$ into cyclic summands, then after possibly restricting to a subsequence, we may assume that the supports of the $p^{n_j} v_j$ in this decomposition are disjoint. Therefore, we can find a decomposition of $V$ into $\oplus_{j<\omega} V_j$ such that $p^{n_j} v_j\in V_j$.

As before, the assignment $p^{n_j}v_{n_j}\mapsto w_j$ clearly extends to a homomorphism $V_j\to W$. Summing extends these all to a homomorphism $\gamma: V\to W$ with $\gamma(p^{n_j}v_{n_j})=w_{n_j}$. Since $X$ is characteristically inert, this contradicts Claim~\ref{sheer}(b), proving the claim.

\vskip .2in

One sees that Claim~\ref{thirdclaim} implies that $p^N X$ is pure in $p^N A$, and since $(p^N X)[p]=S_{\alpha_N}=(p^N A)[p]$ this means that $p^N X=p^N A$. And since $X\subseteq A$, and by Claim~\ref{firstclaim} $X[p^N]\sim A[p^N]$,  by Lemma~\ref{quotient}(a) we can conclude that $X\sim A$, as required.
\end{proof}

If $H$ is separable, then it clearly satisfies the hypotheses of Theorem~\ref{big}. And in that proof, any subgroup that is characteristically inert (and hence also any subgroup that is fully inert) is commensurable with the fully invariant subgroup $G(\overline \alpha)$. We, therefore, have the following consequence.

\begin{corollary}\label{summand} Any separable group is a summand of a separable group with minimal characteristic inertia as well as with minimal full inertia.
\end{corollary}

Observe that the last result holds for the prime $2$, even though a separable 2-group may have characteristic subgroups that fail to be fully invariant.

\medskip

If $p$ is odd, we already showed in Corollary~\ref{sep}, listed above, that any separable group that has minimal characteristic inertia also has minimal full inertia. However, the subsequent Example~\ref{minchar} will demonstrate in the sequel that this does not extend to non-separable groups.

\section{Examples and Problems}

The assertions quoted above allow us to extract an example of a group that has neither minimal characteristic inertia nor minimal full inertia.

\begin{example}\label{noone} For any prime $p$ there exists a $p$-group that has {\it neither} minimal full inertia {\it nor} minimal characteristic inertia.
\end{example}

\begin{proof} Let $B=\oplus_{n<\omega} \mathbb{Z}_{p^{n+1}}$ be the standard direct sum of cyclic groups and let $\overline{B}$ be its torsion completion. In \cite {GS} it was noted that $B\oplus \overline B$ does not have minimal full inertia. In that proof, it was shown that $X = 0 \oplus \overline{B}[p]$ is fully inert but not commensurable with a fully invariant subgroup.

We claim that $B\oplus \overline B$ does not have minimal characteristic inertia, as well. First, since $X$ is fully inert, it follows that it is characteristically inert. So it will suffice to verify that $X$ is not commensurable with a characteristic subgroup.

Let $k<\omega$ be arbitrary. Clearly, there is a decomposition
$$B=B_1\oplus B_2:= \left(\oplus_{n<k} \Z_{p^{n+1}}\right)\oplus \left(\oplus_{k\leq n} \Z_{p^{n+1}}\right)
$$
Define $\phi_k$, an automorphism of $G$, using the decomposition
$$G=(B_1\oplus \overline B_2)\oplus (B_1\oplus \overline B_2)$$
as follows. If $x_1,x_1'\in B_1$ and $x_2,x_2'\in \overline B_2$, let
$$
        \phi_k((x_1,x_2,x_1',x_2'))=(x'_1,x_2,x_1,x_2').
$$
It is now easy to verify that $\hat \phi_k$ is isomorphic to the socle $B_1[p]$, which has order $p^k$.

Letting $k$ vary, we can infer that $X$ is not {\it uniformly} characteristically inert (where that term is defined as in the case of {\it uniformly} fully inert by using automorphisms instead of endomorphisms -- see, for instance, \cite{GS}). Consequently, $X$ is not commensurable with a characteristic subgroup, as stated.
\end{proof}

We observed in Proposition~\ref{minfull}, that if the group $G$ has the property that all of its characteristic subgroups are fully invariant, then if $G$ has minimal characteristic inertia, then it must have minimal full inertia. So, if $G$ is a group with minimal characteristic inertia, but does not have minimal full inertia, then it must have characteristic subgroups that are not fully invariant. The following shows that this can happen.

\begin{example}\label{minchar} There exists a group with minimal characteristic inertia that fails to have minimal full inertia. In fact, for any prime $p$ (even for $p=2$), there is a group of length $\omega+1$ with these properties.
\end{example}

\begin{proof} Let $R$ be the polynomial ring $\Z_p[z]$. Next, let $M$ be the $R$-module $\Z_p[z,z^{-1}]=\{z^kr(z): k\in \mathbb Z, r(z)\in R\}$ (so $M$ is a submodule of the quotient ring of $R$). By a classical result of Corner (\cite C), there is a group $H$ such that $p^\omega H=M$ and the endomorphism ring on $H$ restricts to $R$ on $M$ in such a way that the automorphism group of $H$ restricts to the units of $R$, i.e., the non-zero elements of $\Z_p$.

Let $G=H\oplus K$ be as in Theorem~\ref{big}, so that $G$ has minimal characteristic inertia.

We, therefore, need to show that $G$ does not have minimal full inertia. Let
$$
                 E=\langle 1, z^{-1}, z^{-2}, z^{-3}, \dots\rangle\subseteq M=p^\omega G.
$$
Any endomorphism of $G$ restricted to $M$ is simply multiplication by some polynomial
$$
           r(x)= a_0+a_1 z +\cdots +a_k z^k.
$$
If for $i=0,\dots, k$, $\phi_i$ is the endomorphism given by multiplication by $z^i$, then it easily follows that $\hat\phi_i(E)$ is finite (in fact, it is naturally isomorphic to the $\Z_p$ span of $z, z^2, \dots, z^i$, so that it has dimension $i$). It follows from Lemma~\ref{bound} that $\hat r(E)$ is also finite. Thus, $E$ has full inertia. On the other hand, since $\hat \phi_k(E)\cong \Z_p^k$ for each $k$, $E$ does not have {\it uniform} full inertia as defined in \cite{GS}, so it is not commensurable with a fully invariant subgroup. Therefore, $G$ does not have minimal full inertia, as asserted.
\end{proof}

By combining Corollary~\ref{summand} and Example~\ref{minchar}, we observe that the class of groups with minimal characteristic inertia is {\it not} closed under taking direct summands.

In Example~\ref{noone} it was noted that $B\oplus \overline B$ has neither minimal full inertia nor minimal characteristic inertia. This immediately implies that neither class is closed under taking direct sums. On the other hand, in Theorem~\ref{mincharfull} it was observed that the group $A$ has minimal full inertia if, and only if, its ``square", $G=A\oplus A$, has this property. The last result immediately implies that this does not hold for the property of having minimal characteristic inertia and thus we arrive at the following consequence.

\begin{corollary}
There is a group $A$ with minimal characteristic inertia, but the square $G=A\oplus A$ does not have minimal characteristic inertia.
\end{corollary}

\begin{proof} Let $A$ be any group with minimal characteristic inertia, but not minimal full inertia. \end{proof}

As already noticed above, many examples of groups without minimal full inertia have been given, i.e., groups $G$ with a fully inert subgroup $X$ that is not commensurable with a fully invariant subgroup (for more details, we refer to \cite{GS}). The following, however, gives an example where $X$ is actually characteristic.

\begin{example}
There is a $2$-group $G$ that has a characteristic subgroup $X$ which is fully inert, but {\it not} commensurable with a fully invariant subgroup.
\end{example}

\begin{proof}
Suppose $U=\mathbb Z_2\oplus \mathbb Z_8$ and $S$ is the subgroup $\langle (1,2)\rangle =\{(0,0)$, $(1,2)$, $(0,4)$, $(1,6)\}$; actually, Kaplansky noted in \cite{Kap} that $S$ is characteristic in $U$, but not fully invariant.

We now let $U^\mathbb N=\oplus_{i\in \mathbb N}U_i$, where each $U_i$ is an isomorphic copy of $U$. Consider the collection $R$ of endomorphisms of $U^\mathbb N$ of the form $\phi:=(\phi_i)_{i\in \mathbb N}$, where $\phi_i:U_i\to U_i$ and such that, for some $N\in \mathbb N$ and $m\in \mathbb Z$, the map $\phi$ restricted to $\oplus_{i\geq N}U_i\subseteq U^\mathbb N$ is simply multiplication by $m$.

Using again the classical result of Corner (\cite{C}, Theorem~10.2), we can construct a group $G$ with $2^\omega G=U^{\mathbb N}$ such that the endomorphisms of $G$, when restricted to $U^{\mathbb N}$, are precisely the ring $R$ and whose automorphisms, again when restricted to $U^{\mathbb N}$, are precisely the units of $R$.

Let $X:=S^\mathbb N=\oplus_{i\in \mathbb N} S_i$, where $S_i$ in $U_i$ is simply $S$ in this copy of $U$.

Since $S$ is characteristic in $U$ and all of the endomorphism in $R$ respect the given coordinate structure, it immediately follows that $X$ is actually characteristic in $G$.

We now observe that $X$ is fully inert. If $\phi\in R$ and we restrict $\phi$ to $U^{\mathbb N}$ (using the same letter), we can by construction find an $N\in \mathbb N$ and $m\in \mathbb Z$ such that $\phi$ restricted to $\oplus_{i\geq N}U_i\subseteq U^\mathbb N$ is simply multiplication by $m$. It follows that $\phi(\oplus_{i\geq N}S_i)\subseteq X$.
Therefore,
$$
          \hat \phi(X)=[\phi(\oplus_{i< N}S_i)+X]/X.
$$
Since $\phi(\oplus_{i< N}S_i)$ is finite, it readily follows that $\hat\phi(X)$ is, as well, so that $X$ is fully inert.

On the other hand, consider the idempotent homomorphism $\kappa: U\to U$ given by $\kappa((a,b))=(a,0)$. So, $(1,2)\in S$, but $\kappa ((1,2))=(1,0)\not\in S$. For each $n\in \mathbb N$, consider the endomorphism $\phi_n$ of $G$ that restricts to $\kappa$ on each $U_i$ for $i\leq n$, and is 0 on each $U_i$ for $i>N$. It is easy to see that $|\hat \phi_n(X)|=2^n$ for all $n\in \mathbb N$. Consequently, $X$ is not {\it uniformly} fully inert (see \cite{GS}). But this means that it is not commensurable with a fully invariant subgroup, either.
\end{proof}

In contrast to the statement of the previous example, we may now a little refine it in order to exhibit a $2$-group having a characteristic subgroup that is not fully inert, thus somewhat extending the ingenious example from \cite{C}.

\begin{example} There is a $2$-group $G$ having a characteristic subgroup which is {\it not} a fully inert subgroup.
\end{example}

\begin{proof} Retain the group $U$ and its characteristic subgroup $S$ from the last example. If $R$ is the endomorphism ring of $U$, then associating each $\phi\in R$ with the endomorphism $U^{\mathbb N}\to U^{\mathbb N}$ given by applying $\phi$ to each coordinate individually, we can view $R$ as acting on  $U^{\mathbb N}$.

We again use Corner's Theorem to produce a group $G$ with $2^\omega G=U^{\mathbb N}$ as above. Again, the same $X$ as defined in our previous example is easily seen to be characteristic, but not fully inert.
\end{proof}

In \cite {C}, Corner produced an example of group that was transitive, but not fully transitive.  For $p\not= 2$, any transitive group is always fully transitive, so his example was necessarily a $2$-group. In our final result, we amend his construction to show there is such a group with the property that all of its characteristic subgroups are fully invariant.

\begin{example}\label{transnonfull}
There is a transitive, non-fully transitive $2$-group, all of whose characteristic subgroups are fully invariant.
\end{example}

\begin{proof}
Let $G$ be Corner's original construction of a transitive, but not fully transitive, group. It has the specific property that $2^\omega G$ is finite; in fact, it follows from our previous examples that $2^\omega G=\Z_2\oplus \Z_8=U$. In addition, if $R$ is the collection of endomorphisms of $G$ in Corner's example restricted to $2^\omega G$, then  $R$ is the subring of the endomorphisms of $2^\omega G$, generated by the collection of automorphism of $2^\omega G$, and any such automorphism of $2^\omega G$ extends to an automorphism of $G$.

Since the transitivity or full transitivity of $G$ depends only on how the endomorphism ring acts on $2^\omega G$, adding a separable summand to $G$ will always give another group with the same properties. Therefore, without loss of generality, we may assume that for every $n\in \N$, the $n$th Ulm invariant of $G$ is at least $2$.

\medskip

We need to show that an arbitrary characteristic subgroup $X$ is fully invariant. So, it suffices to show that, if $\gamma$ is any endomorphism of $G$ and $x$ is any element of $X$, then $\gamma(x)\in X$.

To that aim, suppose first that $2^k$ is the order of $x+2^\omega G\in G/2^\omega G=G_1$; so $|2^j x|<\omega$ if, and only if, $j<k$.

We know that $\gamma$ restricted to $2^\omega G$ is a sum of automorphisms of $2^\omega G$, say $\phi_1+\cdots+ \phi_m$, and that all of these maps can be extended to automorphisms of $G$. Since $2^k x\in 2^\omega G$, we have
$$
                        \gamma(2^kx)= \phi_1(2^k x)+\cdots+ \phi_m(2^k x).
$$
If $\gamma'=\gamma-(\phi_1+\cdots+ \phi_m)$, then since $X$ is characteristic, $\gamma(x)\in X$ if, and only if, $\gamma'(x)\in X$. Replacing $\gamma$ by $\gamma'$, there is no loss of generality in assuming that $\gamma(2^kx)=0$.

Let $z=\gamma(x)$; so $2^k z=0$ and we clearly have an inequality of height sequences $\Vert x\Vert\leq \Vert z\Vert$.  Set $y_j=2^j x$ and $y'_j=-2^j x$ for all $j\geq k$. By a technique that goes back to Theorem~26 of \cite{Kap} (which also appears in Theorem~2.13 of \cite {CDK} and was mentioned earlier in this work), we can induct backwards from $j=k$ down to $j=0$ to construct $y_{k-1}, \dots, y_1, y_0$ and $y'_{k-1}, \dots, y'_1, y'_0$ such that, for each $0\leq j$, we have:

\medskip

(1) $2^j z= y_j+y'_j$, i.e., $2^j z-y_j=y'_j$;

\medskip

(2) $2y_j=y_{j+1}$ and $2y_j'=y_{j+1}'$;

\medskip

(3) $|y_j|=|y'_j|=|2^j x|$.

\medskip

In fact, these conditions clearly hold for all $j\geq k$, so assume they hold down to $j+1$ and we need to construct $y_j$, which by (1) will define the wanted $y_j'$.

\medskip

{\it Case 1}: $|2^j x|+1 = |2^{j+1} x|=|y_{j+1}|$: Let $y_j$ be any element of $G$ satisfying $2y_j=y_{j+1}$ and $|y_j|=|2^j x|$. By (1), we must let $y_j'=2^j z-y_j$ and (2) follows easily.

Certainly, $|y_j'|\geq \min\{ |y_j|, |2^j z|\}=|2^j x|$. And if $|y_j'|> |2^j x|$, then
$$ |2^j x|+1=|y_{j+1}|=|y_{j+1}'|= |2y_j'|\geq |y_j'|+1>|2^j x|+1.$$
This contradiction verifies (3) for $j$.

\medskip

{\it Case 2}: $|2^j x|+1 < |2^{j+1} x|=|y_{j+1}|$: Let $n=|2^j x|$. We start by finding $s\in G$ such that $|s|\geq n+1$ and $2s=y_{j+1}$.

Recall that the $n$th Ulm factor of $G$ is isomorphic to
$$
            U_n:=\{ g\in 2^n G: 2g\in 2^{n+2}G\}/ 2^{n+1} G.
$$
Since $|2^j z|\geq |2^j x|= n$ and $|2^{j+1} z|\geq |2^{j+1} x|\geq n+2$, we see that $v:= 2^j z+2^{n+1} G$ represents a (possibly zero) element of $U_n$.

Since $U_n$ has at least $3$ elements, it has a non-zero element $w\ne v$. Using the usual way to think of Ulm factors, we can find $t\in (2^n G)[2]$ such that $t+2^{n+1} G= w$; since $w\ne 0$, we have $|t|=n$.

Let $y_j=s+t$. So, $ y_j+2^{n+1} G=t+2^{n+1}G=v$ and $2y_j =2s+2t=2s=y_{j+1}$. To make sure (1) continues to hold, we must define $y_j'=2^j z-y_j$, which will again imply that $2y_j'=y_{j+1}'$, so that (2) will hold as well.

Regarding (3), since $|s|\geq n+1$, we can conclude that $|y_j|=|s+t|=n=|2^j x|$, which is half the battle. In addition, since $|y_j|= n$ and $|2^j z|\geq |2^j x|=n$, we can deduce that $|y_j'|\geq n$. Finally, since
$$
             y'_j+2^{n+1} G= (2^j z+2^{n+1} G)- (y_j+2^{n+1} G)=v-w\in U_n
$$
is non-zero, we can conclude that $|y_j'|=n=|2^j x|$, as required.

\medskip

Setting $j=0$, we can infer from (1) that $z=y_0+y'_0$. And from (2) and (3), the height sequences must satisfy
$\Vert y_0\Vert =\Vert y_0' \Vert =\Vert x \Vert$. So, since $G$ is transitive, there are automorphisms $\alpha$ and $\alpha'$ such that $\alpha(x)=y_0$ and $\alpha'(x)=y_0'$. Therefore, since $X$ is characteristic, one finds that
$$
                     z=y_0+y_0'= \alpha(x)+\alpha'(x)\in X,
$$
as required.
\end{proof}

A query which immediately arises is of whether or not the Krylov transitive $2$-group as constructed in \cite{BGGS} to be neither transitive nor fully transitive has the same property as in the preceding example, that is, are all its characteristic subgroups fully invariant, or even commensurable with fully invariant subgroups? Moreover, it is rather logical to have true the assertion that, for all primes $p$, any Krylov transitive $p$-group whose characteristic subgroups are (commensurable with) fully invariant subgroups is necessarily transitive.

\medskip

We end our work with two problems of some interest. The following question is clearly important and possibly difficult (compare with Example~\ref{minchar}).

\medskip

\noindent{\bf Problem 1.} Does every group with minimal full inertia also have minimal characteristic inertia?

\medskip

In case this is not true, that is these two properties are independent each other, we proceed with the next possibly challenging question, which is relevant to Corollary~\ref{totproj}.

\medskip

\noindent{\bf Problem 2.} Do totally projective $2$-groups have the minimal characteristic inertia?

\medskip

It is worthwhile noticing that, in view of \cite{K}, they always have the minimal full inertia.

\medskip

\noindent {\bf Funding:} The work of the first-named author, P.V. Danchev, is partially supported by the Bulgarian National Science Fund under Grant KP-06 No. 32/1 of December 07, 2019.

\end{document}